\documentclass[12pt]{article}
\usepackage{amssymb}
\usepackage{amsmath,amsthm}
\usepackage[latin1]{inputenc}
\usepackage[dvips]{graphicx}
\usepackage{hyperref}
\usepackage{enumerate}

\hypersetup{colorlinks=true}

\hypersetup{colorlinks=true, linkcolor=blue, citecolor=blue,urlcolor=blue}

 \newtheorem{remark}{Remark}

 \newtheorem{theorem}[remark]{Theorem}
 \newtheorem{proposition}[remark]{Proposition}
 \newtheorem{corollary}[remark]{Corollary}
  \newtheorem{claim}[remark]{Claim}

\title{On the strong partition dimension of graphs}

\author{Ismael Gonz\'alez Yero\\
\\
{\small Departamento de Matem\'aticas, Escuela Polit\'ecnica Superior de Algeciras}\\
{\small Universidad de C\'adiz,} {\small
Av. Ram\'on Puyol s/n, 11202 Algeciras, Spain.} \\ {\small
ismael.gonzalez\@@uca.es}
}

\date{June 10, 2013}

\begin{document}

\maketitle

\begin{abstract}
We present a different way to obtain generators of metric spaces having the property that the ``position'' of every element of the space is uniquely
determined by the distances from the elements of the generators. Specifically we introduce a generator based on a partition of the metric space into sets of elements. The sets of the partition will work as the new elements which will uniquely determine the position of each single element of the space. A set $W$ of vertices of a connected graph $G$ strongly resolves two different vertices $x,y\notin W$ if either $d_G(x,W)=d_G(x,y)+d_G(y,W)$ or $d_G(y,W)=d_G(y,x)+d_G(x,W)$, where $d_G(x,W)=\min\left\{d(x,w)\;:\;w\in W\right\}$. An ordered vertex partition $\Pi=\left\{U_1,U_2,...,U_k\right\}$ of a graph $G$ is a strong resolving partition for $G$ if every two different vertices of $G$ belonging to the same set of the partition are strongly resolved by some set of $\Pi$. A strong resolving partition of minimum cardinality is called a strong partition basis and its cardinality the strong partition dimension. In this article we introduce the concepts of strong resolving partition and strong partition dimension and we begin with the study of its mathematical properties. We give some realizability results for this parameter and we also obtain tight bounds and closed formulae for the strong metric dimension of several graphs.
\end{abstract}



\section{Introduction}

A vertex $v\in V$ is said to distinguish two vertices $x$ and $y$ if $d_G(v,x)\ne d_G(v,y)$, where $d_G(x, y)$ is the length of a shortest path between $x$ and $y$. A set $S\subset V$ is said to be a \emph{metric generator} for $G$ if any pair of vertices of $G$ is distinguished by some element of $S$. A minimum generator is called a \emph{metric basis}, and its cardinality the \emph{metric dimension} of $G$, denoted by $dim(G)$. Motivated by the problem of uniquely determining the location of an intruder in a network, the concept of locating set was introduced by Slater in \cite{leaves-trees}. The concept of resolving set of a graph was also introduced by Harary and Melter in \cite{harary}, where locating sets were called resolving sets. In fact, the concepts of locating set and resolving set coincides with the concept of metric generator for the metric space $(G, d_G)$, where $G =(V, E)$ is a connected graph, $d_G : V\times V\rightarrow \mathbf{N}$. In this sense, locating sets, resolving sets and metric generators represent the same structure in a graph $G$. Throughout the article $G=(V,E)$ denotes a simple graph or order $n=|V(G)|$, minimum degree $\delta(G)$ and maximum degree $\Delta(G)$ ($\delta$ and $\Delta$ for short).

Slater described the usefulness of these ideas into long range aids to navigation \cite{leaves-trees}. Also, these concepts have some applications in chemistry for representing chemical compounds \cite{pharmacy1,pharmacy2} or in problems of pattern recognition and image processing, some of which involve the use of hierarchical data structures \cite{Tomescu1}. Other applications of this concept to navigation of robots in networks and other areas appear in
\cite{chartrand,robots,landmarks}. Hence, according to its applicability resolving sets has became into an interesting and popular topic of investigation in graph theory. While applications have been continuously appearing, also this invariant has been theoretically studied in a high number of other papers including for example, \cite{chappell,chartrand,chartrand2,fehr,Tomescu1,LocalMetric,tomescu}. Moreover, several variations of metric generators including resolving dominating sets \cite{brigham}, independent resolving sets \cite{chartrand1}, local metric sets \cite{LocalMetric}, strong resolving sets \cite{may-oellermann,Oellermann,seb}, metric colorings \cite{metric-col}  and  resolving partitions \cite{chappell,chartrand2,fehr,tomescu,note-pd-cart}, etc. have been introduced and studied.

Strong metric generators in metric spaces or graphs were first described in \cite{seb}, where the authors presented some applications of this concept to combinatorial search. For instance they worked with problems on false coins known from the borderline of extremal combinatorics and information theory and also, with a problem known from combinatorial optimization related to finding ``connected joins'' in graphs. In such a work results about detection of false coins are used to approximate the value of the metric dimension of some specific graphs. They also proved that the existence of connected joins in graphs can be solved in polynomial time, but on the other hand they obtained that the minimization of the number of components of a connected join is NP-hard. Metric generators were also studied in \cite{Oellermann} where the authors found an interesting connection between the strong metric basis of a graph and the vertex cover of a related graph which they called ``strong resolving graph''. This connection allowed them to prove that finding the metric dimension of a graph is NP-complete. Nevertheless when the problem is restricted to trees, it can be solved in polynomial time, fact that was also noticed in \cite{seb}. A remarkable article about strong metric generators is \cite{strong-genetic}, where the authors used some genetic algorithms to compute the strong metric dimension of some classes of graphs. Other examples of works about strong metric generators are for instance \cite{strong-hamming,strong-polytopes,may-oellermann,Oellermann,seb}.

In this article we present a different way to obtain generators of metric spaces maintaining the property of other similar well known generators related to that the  ``position'' of every element of the space is uniquely determined by the distances from the elements (landmarks) of the generators. Specifically here we introduce a generator based on a partition $\Pi$ of the metric space into sets (set marks). The set marks (sets of the partition $\Pi$) will work as the elements which will uniquely determine the position of each single element of the space. Some of the principal antecedents of this new generator is the resolving partition defined in \cite{chartrand2} or the metric colorings presented in \cite{metric-col}.

A vertex $v$ of a graph $G$ {\em strongly resolves} the two different vertices $x,y$ if either $d_G(x,v)=d_G(x,y)+d_G(y,v)$ or $d_G(y,v)=d_G(y,x)+d_G(x,v)$.
A set $S$ of vertices in a connected graph $G$ is a \emph{strong resolving set} for $G$ if every two vertices of $G$ are strongly resolved by some vertex
of $S$. A strong resolving set of minimum cardinality is called a {\em strong metric basis} and its cardinality the \emph{strong metric dimension} of $G$, which is denoted by $dim_s(G)$.

For a vertex $x$ and a set $W$ of $G$ it is defined the distance between $x$ and $W$ in $G$ as $d_G(x,W)=\min\left\{d(x,w)\;:\;w\in W\right\}$ (if the graph is clear from the context we use only $d(x,W)$). Now notice that given two different vertices $x,y$ and a set of vertices $A$, such that $x\in A$ and $y\notin A$, since $d(x,A)=0$, it could happen that $\min\left\{d(y,a):\,a\in A\right\}=d(y,A)\ne d(y,x)+d(x,A)$. Hence, in order to avoid that case we say that a set $W$ of vertices of $G$ {\em strongly resolves} two different vertices $x,y\notin W$ if either $d_G(x,W)=d_G(x,y)+d_G(y,W)$ or $d_G(y,W)=d_G(y,x)+d_G(x,W)$.  An ordered vertex partition $\Pi=\left\{U_1,U_2,...,U_k\right\}$ of a graph $G$ is a {\em strong resolving partition} for $G$ if every two different vertices of $G$ belonging to the same set of the partition are strongly resolved by some set of $\Pi$. A strong resolving partition of minimum cardinality is called a {\em strong partition basis} and its cardinality the \emph{strong partition dimension}, which is denoted by $pd_s(G)$. Notice that always $pd_s(G)\ge 2$.

A vertex $u$ of $G$ is \emph{maximally distant} from $v$ if for every vertex $w$ in the open neighborhood of $u$, $d_G(v,w)\le d_G(u,v)$. If $u$ is maximally distant from $v$ and $v$ is maximally distant from $u$, then we say that $u$ and $v$ are \emph{mutually maximally distant}.  The {\em boundary} of $G=(V,E)$ is defined as $\partial(G) = \left\{u\in V: \mbox{ there exists $v\in V$  such that $u,v$  are mutually maximally distant }\right\}$. For some basic graph classes, such as complete graphs $K_n$, complete bipartite graphs $K_{r,s}$,  cycles $C_n$ and hypercube graphs $Q_k$, the boundary is simply the whole vertex set.
It is not difficult to see that this property holds for all  $2$-antipodal\footnote{The diameter of $G=(V,E)$ is defined as $D(G)=\max_{u,v\in V}\left\{d(u,v)\right\}$.  We recall that $G=(V,E)$ is $2$-antipodal if for each vertex $x\in V$ there exists exactly one vertex $y\in V$ such that $d_G(x,y)=D(G)$.} graphs and also for all distance-regular graphs.
Notice that the boundary of a tree consists exactly of the set of its leaves. A vertex  of a graph  is a {\em simplicial vertex} if the subgraph induced by  its neighbors is a complete graph. Given a graph $G$, we denote by $\varepsilon(G)$ the set of simplicial vertices of $G$. If the simplicial vertex has degree one, then it is called an {\em end-vertex}. We denote by $\tau(G)$ the set of end-vertices of $G$.
Notice that $\sigma(G)\subseteq \partial(G)$.

The notion of strong resolving graph was introduced first in \cite{Oellermann}. The \emph{strong resolving graph}\footnote{In fact, according to \cite{Oellermann} the strong resolving graph $G'_{SR}$ of a graph $G$ has vertex set $V(G'_{SR})=V(G)$ and two vertices $u,v$ are adjacent in $G'_{SR}$ if and only if $u$ and $v$ are mutually maximally distant in $G$. So, the strong resolving graph defined here is a subgraph of the strong resolving graph defined in \cite{Oellermann} and can be obtained from the latter graph by deleting its isolated vertices.} of $G$ is a graph $G_{SR}$  with vertex set $V(G_{SR}) = \partial(G)$ where two vertices $u,v$ are adjacent in $G_{SR}$ if and only if $u$ and $v$ are mutually maximally distant in $G$. There are some families of graph for which its strong resolving graph can be obtained relatively easy. For instance, we emphasize the following cases.

\begin{remark}$\,$\label{known-G_SR}
\begin{enumerate}[{\rm (i)}]
\item If $\partial(G)=\sigma(G)$, then $G_{SR}\cong K_{|\partial(G)|}$. In particular, $(K_n)_{SR}\cong K_n$ and for any tree $T$ with $l(T)$ leaves, $(T)_{SR}\cong K_{l(T)}$.

\item For any connected block graph\footnote{$G$ is a block graph if every biconnected component (also called block) is a clique. Notice that any vertex in a block graph is either a simplicial vertex or a cut vertex.} $G$ of order $n$ and $c$ cut vertices, $G_{SR}\cong K_{n-c}$.

\item For any $2$-antipodal graph $G$ of order $n$, $G_{SR}\cong \bigcup_{i=1}^{\frac{n}{2}} K_2$. In particular,  $(C_{2k})_{SR}\cong \bigcup_{i=1}^{k} K_2$ and for any hypercube graph $Q_r$, $(Q_r)_{SR}\cong \bigcup_{i=1}^{2^{r-1}} K_2$ .

\item For any positive integer $k$, $(C_{2k+1})_{SR}\cong C_{2k+1}$.
\end{enumerate}
\end{remark}

\section{Realizability and basic results}

Clearly the strong metric dimension and the strong partition dimension are related. If $S=\left\{v_1,v_2,...,v_r\right\}$ is a strong metric basis of $G=(V,E)$, then it is straightforward to observe that the partition $\Pi=\left\{\left\{v_1\right\},\left\{v_2\right\},...,\left\{v_r\right\},V-S\right\}$ is a strong resolving partition for $G$. Thus the following result is obtained.

\begin{theorem}\label{bound-dim-pd}
For any connected graph $G$, $pd_s(G)\le dim_s(G)+1$.
\end{theorem}

A set $S$ of vertices of $G$ is a \emph{vertex cover} of $G$ if every edge of $G$ is incident with at least one vertex of $S$. The \emph{vertex cover number} of $G$, denoted by $\alpha(G)$, is the smallest cardinality of a vertex cover of $G$. We refer to an $\alpha(G)$-set in a graph $G$ as a vertex cover  of cardinality $\alpha(G)$. Oellermann and Peters-Fransen \cite{Oellermann} showed that the problem of finding the strong metric dimension of a connected graph $G$ can be transformed to the problem of finding the vertex cover number of $G_{SR}$.

\begin{theorem}{\em \cite{Oellermann}}\label{lem_oellerman}
For any connected graph $G$,
$dim_s(G) = \alpha(G_{SR}).$
\end{theorem}

While an equivalent result for the strong partition dimension of graphs could not be deduced, the result leads to an upper bound for the strong partition dimension. According to Theorems \ref{bound-dim-pd} and \ref{lem_oellerman} we have the following result.

\begin{theorem}\label{bound-cover-pd}
For any connected graph $G$, $pd_s(G)\le \alpha(G_{SR})+1$.
\end{theorem}

A {\em clique} in a graph $G$ is a set of vertices $S$ such that $\langle S\rangle$ is isomorphic to a complete graph. The maximum cardinality of a clique in a graph $G$ is the {\em clique number} and it is denoted by $\omega(G)$. We will say that $S$ is an $\omega(G)$-clique if $|S|=\omega(G)$. The following claim, and its consequence, will be very useful to obtain the strong metric dimension of some graphs.

\begin{claim}
If two vertices are mutually maximally distant in a graph $G$, then they belong to different sets in any strong resolving partition for $G$.
\end{claim}

\begin{corollary}\label{coro-pd-clique}
For any connected graph $G$, $pd_s(G)\ge \omega(G_{SR})$.
\end{corollary}

Throughout the article we will present several examples (for instance paths, trees, complete graphs) in which the bounds of Theorems \ref{bound-dim-pd} and \ref{bound-cover-pd}, and Corollary \ref{coro-pd-clique} are tight. Nevertheless, the bounds of Theorems \ref{bound-dim-pd} and \ref{bound-cover-pd} are frequently very far from the exact value for the strong partition dimension. In fact, the difference between the strong metric dimension and the strong partition dimension of a graph can be arbitrarily large as we show at next. To do so we need to introduce some additional notation. Given a unicyclic graph $G$ with unique cycle $C_t$, any end-vertex $u$ of $G$ is said to be a \emph{terminal vertex} of a vertex $v\in V(C_t)$ if $d_{G}(u,v)<d_{G}(u,w)$ for every other vertex $w\in V(C_t)$. The {\em terminal degree} $ter(v)$ of a vertex $v\in V(C_t)$ is the number of terminal vertices of $v$. A vertex $v\in V(C_t)$ of $G$ is a \emph{major vertex} of $G$ if it has positive terminal degree ($v$ has degree greater than two in $G$). Notice that $|\tau(G)|=\sum_{v_i\in V(C_t)}ter(v_i)$.

Let $\mathcal{C}_1$ be the family of unicyclic graphs $G(r,t)$, $r\ge 2$ and $t\ge 4$, defined in the following way. Every $G(r,t)\in \mathcal{C}_1$ has unique cycle $C_t$ of order $t\ge 4$ and also, $G(r,t)$ has only one vertex $v\in V(C_t)$ of degree greater than two with $ter(v)=\delta(v)-2=|\tau(G)|=r\ge 2$.

\begin{proposition}\label{prop-C_1}
If $G(r,t)\in \mathcal{C}_1$, then
$$pd_s(G(r,t))=r+1\; \mbox{ and }\; dim_s(G(r,t))=\left\{\begin{array}{ll}
    r+\frac{t-1}{2}, & \mbox{ if $t$ is odd,} \\
    & \\
    r+\frac{t-2}{2}, & \mbox{ if $t$ is even.}
  \end{array}\right.
$$
\end{proposition}

\begin{proof}
Let $V(C_t)=\left\{v_0,v_1,...,v_{t-1}\right\}$. Without loss of generality we suppose that $v_0$ is the vertex of $G(r,t)$ such that $\delta(v_0)=r+2$. Let $u_1,u_2,...,u_r$ be the set of terminal vertices of $v_0$ and let $P_i$ be the $u_i-v_0$ path for every $i\in \left\{1,...,r\right\}$. Notice that $u_1,u_2,...,u_r$ are mutually maximally distant between them and also, there exists at least one vertex (for instance $v_{\left\lfloor\frac{t}{2}\right\rfloor}$) being diametral with $v_0$ in $C_t$ such that it is mutually maximally distant with any other vertex in $u_1,u_2,...,u_r$. Thus, $pd_s(G(r,t))\ge \omega((G(r,t))_{SR})\ge r+1$.

Now, suppose $t$ even. Let the vertex partition $\Pi=\left\{A_1,...,A_r,B\right\}$ such that $A_1=V(P_1)\cup\left\{v_1,v_2,...,v_{\left\lfloor\frac{t}{2}\right\rfloor-1}\right\}$, $A_2=V(P_2)-\left\{v_0\right\}$, $A_3=V(P_3)-\left\{v_0\right\}$, ..., $A_r=V(P_r)-\left\{v_0\right\}\cup \left\{v_{\left\lfloor\frac{t}{2}\right\rfloor+1},...,v_{t-1}\right\}$ and $B=\left\{v_{\left\lfloor\frac{t}{2}\right\rfloor}\right\}$. We claim that $\Pi$ is a strong resolving partition for $G(r,t)$. Let $x,y$ be two different vertices of $G(r,t)$. If $x,y\in A_1$ or $x,y\in A_r$, then since $d(u_1,v_{\left\lfloor\frac{t}{2}\right\rfloor})=d(u_1,v_0)+d(v_0,v_{\left\lfloor\frac{t}{2}\right\rfloor})=d(u_1,v_0)+D(C_t)$ and $d(u_r,v_{\left\lfloor\frac{t}{2}\right\rfloor})=d(u_r,v_0)+d(v_0,v_{\left\lfloor\frac{t}{2}\right\rfloor})=d(u_r,v_0)+D(C_t)$, we have either $d(x,B)=d(x,v_{\left\lfloor\frac{t}{2}\right\rfloor})=d(x,y)+d(y,v_{\left\lfloor\frac{t}{2}\right\rfloor})=d(x,y)+d(y,B)$ or $d(y,B)=d(y,v_{\left\lfloor\frac{t}{2}\right\rfloor})=d(y,x)+d(x,v_{\left\lfloor\frac{t}{2}\right\rfloor})=d(y,x)+d(x,B)$. Now, if $x,y\in A_i$ for some $i\in\left\{2,...,r-1\right\}$, then we have either $d(x,A_1)=d(x,v_0)=d(x,y)+d(y,v_0)=d(x,y)+d(y,A_1)$ or $d(y,A_1)=d(y,v_0)=d(y,x)+d(x,v_0)=d(y,x)+d(x,A_1)$. Thus, $\Pi$ is a strong resolving partition for $G(r,t)$ and $pd_s(G(r,t))\le r+1$. Therefore, we obtain that $pd_s(G(r,t))=r+1$.

Now suppose $t$ odd and let $\Pi'=\left\{A'_1,...,A'_r,B'\right\}$ such that $A'_1=V(P_1)\cup\left\{v_1,v_2,...,v_{\left\lfloor\frac{t}{2}\right\rfloor-1}\right\}$, $A'_2=V(P_2)-\left\{v_0\right\}$, $A'_3=V(P_3)-\left\{v_0\right\}$, ..., $A'_r=V(P_r)-\left\{v_0\right\}\cup \left\{v_{\left\lceil\frac{t}{2}\right\rceil+1},...,v_{t-1}\right\}$ and $B'=\left\{v_{\left\lfloor\frac{t}{2}\right\rfloor},v_{\left\lceil\frac{t}{2}\right\rceil}\right\}$. Analogously to the case $t$ even we obtain that $\Pi'$ is a strong resolving partition for $G(r,t)$ and $pd_s(G(r,t))\le r+1$. Therefore, we obtain that $pd_s(G(r,t))=r+1$.

To obtain the strong metric dimension of $G(r,t)$ we consider the following. If $t$ is even, then the strong resolving graph of $G(r,t)$ is formed by $\frac{t-2}{2}+1$ connected components, one of them isomorphic to a complete graph $K_{r+1}$ and the other $\frac{t-2}{2}$ copies isomorphic to $K_2$. Thus we have that $dim_s(G(r,t))=\alpha((G(r,t))_{SR})=r+\frac{t-2}{2}$. On the other hand, if $t$ is odd, then the strong resolving graph of $G(r,t)$ is isomorphic to a graph obtained from a complete graph $K_r$ and a path $P_{t-1}$ by adding all the possible edges between the vertices of the complete graph $K_r$ and the leaves of the path $P_{t-1}$. Thus we have that $dim_s(G(r,t))=\alpha((G(r,t))_{SR})=r+\frac{t-1}{2}$.
\end{proof}

From now on in this section we deal with the problem of realization for the strong partition dimension of graphs.

\begin{theorem}\label{realiz-1}
For any integers $r,n$ such that $2\le r\le n$ there exists a connected graph $G$ of order $n$ with $pd_s(G)=r$.
\end{theorem}

\begin{proof}
If $r=n$, then $pd_s(K_n)=n$. Suppose $r<n$. Let $G$ be a graph defined in the following way. We begin with a complete graph $K_{r}$ and a path of order $n-r+1$. Then we identify one leaf $u$ of the path with a vertex $v$ of the complete graph (in some literature this graph is called a comet). Notice that the order of $G$ is $n$. Also notice that only those vertices of the complete graph different from $v$ and the leaf of the path different from $u$ are mutually maximally distant between them and they form the boundary of $G$. Thus, $G_{SR}\cong K_{r}$ and from Theorem \ref{bound-cover-pd} and Corollary \ref{coro-pd-clique} we obtain that $pd_s(G)=r$.
\end{proof}

As the above results shows, any two pair of integers $r,n$ such that $2\le r\le n$ are realizable as the strong partition dimension and the order of a graph, respectively. Nevertheless, as we can see at next, not every three  integers $r,t,n$ are realizable as the strong partition dimension, the strong metric dimension and the order of a graph, respectively. Since $pd_s(G)=2$ if and only if $G$ is the path $P_n$ (see Theorem \ref{path-only-two}), and $dim_s(P_n)=1$, we have that if $t\ne 1$, then the values $2,t,n$ are not realizable as the strong partition dimension, the strong metric dimension and the order of a graph, respectively. Moreover, since $pd_s(G)=n$ if and only if $G$ is a complete graph $K_n$ (see Theorem \ref{pd-n-complete}), and $dim_s(K_n)=n-1$, it follows that if $r<t+1=n$, then the values $r,t,n$ are not realizable as the strong partition dimension, the strong metric dimension and the order of a graph, respectively.

Notice that the comet graph of the proof of Theorem \ref{realiz-1} has order $n$, $pd_s(G)=r$ and $dim_s(G)=r-1=t<t+1<n$. Thus we have the following result.

\begin{remark}
For any integers $r,t,n$ such that $3\le r=t+1<n$ there exists a connected graph $G$ of order $n$ with $pd_s(G)=r$ and $dim_s(G)=t$.
\end{remark}

\begin{remark}
For any integers $r,t,n$ such that $3\le r=t\le n-3$ there exists a connected graph $G$ of order $n$ with $pd_s(G)=r$ and $dim_s(G)=t$.
\end{remark}

\begin{proof}
To prove the result we consider a graph $H$ constructed in the following way. We begin with a graph $G(r-1,4)\in \mathcal{C}_1$. Suppose that $v$ is the vertex of degree $r+1$ in the unique cycle of $G(r-1,4)$ and let $u$ be a terminal vertex of $v$. Then to obtain the graph $H$ we subdivide the edge $uv$ by adding $n-r-3$ vertices. Notice that the order of $H$ is $n-r-3+r-1+4=n$ and according to Proposition \ref{prop-C_1} we have that $pd_s(H)=r$ and $dim_s(H)=r-1+\frac{4-2}{2}=r=t$.
\end{proof}

\begin{theorem}
For any integers $r,t,n$ such that $3\le r<t\le \frac{n+r-2}{2}$ there exists a connected graph $G$ of order $n$ with $pd_s(G)=r$ and $dim_s(G)=t$.
\end{theorem}

\begin{proof}
To prove the result we will construct a graph $H$ in the following way. We begin with a graph $G(r-1,2(t-r+1)+1)\in \mathcal{C}_1$. Suppose that $v$ is the vertex of degree $r+1$ in the unique cycle of $G(r-1,2(t-r+1)+1)$ and let $u$ be a terminal vertex of $v$. Then to obtain the graph $H$ we subdivide one of the edges of the shortest $u-v$ path by adding $n-2t+r-2$ vertices. Now, notice that the order of $H$ is $n-2t+r-2+2(t-r+1)+1)+r-1=n$ and according to Proposition \ref{prop-C_1} we have that $pd_s(H)=r$ and $dim_s(H)=r-1+\frac{2(t-r+1)+1-1}{2}=t$.
\end{proof}

The above result immediately brings up another question. Is it the case that $dim_s(G)\le \displaystyle\frac{pd_s(G)+n-2}{2}$ for every nontrivial connected graph $G$ of order $n$? For instance, given a graph $G$ of order $n$, if $dim_s(G)\le \frac{n}{2}$, then we have that $2\cdot dim_s(G)-pd_s(G)\le n-pd_s(G)\le n-2$, which leads to $dim_s(G)\le \displaystyle\frac{pd_s(G)+n-2}{2}$.

\section{Exact values for the strong partition dimension of some families of graphs}

We first notice that Remark \ref{known-G_SR}, Theorem \ref{bound-cover-pd} (or Theorem \ref{bound-dim-pd}) and Corollary \ref{coro-pd-clique} lead to the following results.

\begin{theorem}\label{easy-results}
If $\partial(G)=\sigma(G)$, then $pd_s(G)=|\partial(G)|$. In particular,
\begin{enumerate}[{\rm (i)}]
\item for any positive integer $n$, $pd_s(K_n)= n$,
\item for any tree $T$ with $l(T)$ leaves, $pd_s(T)=l(T)$,
\item for any connected block graph $G$ of order $n$ and $c$ cut vertices, $pd_s(G)=n-c$.
\end{enumerate}
\end{theorem}

We continue with a remark which will be useful to present other results. We consider a vertex partition $P(r,t)=\{A_1,A_2,...,A_r,B_1,B_2,...,B_t\}$ of the vertex set of a graph such that every $A_i$ induces a shortest path $a_{i1}\sim a_{i2}\sim...\sim a_{ir_i}$ in $G$ and every $B_i$ induces an isolated vertex. Hence, it is straightforward to observe that the partition $\Pi(r,t)=\{B_1,B_2,...,B_t,A_1-\{a_{11}\},A_2-\{a_{21}\},...,A_r-\{a_{r1}\},\{a_{11}\},\{a_{21}\},...,\{a_{r1}\}\}$ is a strong resolving partition for $G$ of cardinality $2r+t$. Now, a partition $\Pi(r,t)$ is a $\mathcal{P}_1$-partition for $G$ if it is satisfied that $2r+t$ has a minimum value among all possible $P(r,t)$ partitions of $G$. Thus, we have the following result.

\begin{remark}\label{P_1-partition}
Let $G$ be a connected graph and let $\Pi(r,t)$ be a $\mathcal{P}_1$-partition for $G$. Then $pd_s(G)\le 2r+t$.
\end{remark}

\begin{corollary}\label{bound-n-diam}
For any connected graph $G$ of order $n$ and diameter $d$, $pd_s(G)\le n-d+1$.
\end{corollary}

Notice that the above bound is tight. For instance, for complete graphs, path graphs and star graphs. Also, note that if a graph $G$ has order $n$ and $pd_s(G)=n-1$, then as a consequence of Corollary \ref{bound-n-diam} it follows that $D(G)=2$. On the other hand, if $G$ has a vertex partition $\Pi(r)=\{A_1,A_2,...,A_r,B\}$ such that $B$ is an isolated vertex and for every $i\in \{1,...,r\}$, $A_i\cup B$ induces a shortest path in $G$, then it is straightforward to observe that the partition $\Pi(r)$ is a strong resolving partition for $G$ of cardinality $r+1$. A partition $\Pi(r)$ of minimum cardinality in $G$ is a $\mathcal{P}_2$-partition for $G$. Thus, we have the following result.

\begin{remark}\label{P_2-partition}
Let $G$ be a connected graph. If $G$ has a $\mathcal{P}_2$-partition, then $pd_s(G)\le r+1$.
\end{remark}

Notice that there are several graphs having such kind of partition. For instance, star graphs with subdivided edges and the sphere graphs $S_{k,r}$ ($k,r\ge 2$), where $S_{k,r}$ is a graph  defined as follows: we consider $r$ path graphs of order $k+1$ and we identify one extreme of each one of the $r$ path graphs in one pole $a$ and all the other extreme vertices of the paths in a pole $b$. In particular, $S_{k,2}$ is a cycle graph. The case of cycle graphs also shows that the bound is tight (see Proposition \ref{pd-cycle}).

Next we characterize the families of graphs achieving some specific values for the strong partition dimension.

\begin{theorem}\label{pd-n-complete}
Let $G$ be a connected graph of order $n$. Then $pd_s(G)=n$ if and only if $G$ is a complete graph.
\end{theorem}

\begin{proof}
If $pd_s(G)=n$, then by Corollary \ref{bound-n-diam} we have that $D(G)=1$ and, as a consequence, $G$ is a complete graph.
\end{proof}

Next result is useful to give a characterization of graphs of order $n$ having strong partition dimension $n-1$.

\begin{remark}\label{C_5-o-sphere-n-2}
Let $G$ be a connected graph of order $n$. If $G$ has an induced subgraph isomorphic to $C_5$ or $S_{2,3}$ or $K_1+C_4$, then $pd_s(G)\le n-2$.
\end{remark}

\begin{proof}
Since the subgraphs considered has order five, we suppose $\{u_1,u_2,....,u_{n-5}\}$ are the vertices of $G$ not belonging to the corresponding subgraph.

Let the cycle $C_5=v_1v_2...v_5v_1$ be an induced subgraph in $G$. Hence, it is straightforward to observe that the vertex partition $\Pi=\{\{v_1\},\{v_2,v_3\},\{v_4,v_5\},\{u_1\},\{u_2\},...,\{u_{n-5}\}\}$ is a strong resolving partition for $G$ of cardinality $n-2$.

Let $xy_1z$, $xy_2z$ and $xy_3z$ be the paths of the induced subgraph $S_{2,3}$ in $G$. As above the vertex partition $\Pi=\{\{y_2\},\{y_1,x\},\{y_3,z\},\{u_1\},\{u_2\},...,\{u_{n-5}\}\}$ is a strong resolving partition for $G$ of cardinality $n-2$.

Let $C_4=x_1x_2x_3x_4x_1$ be an induced subgraph in $G$ and let $v$ be a vertex of $G$, such that $v\ne x_1,x_2,x_3,x_4$ and $\langle \{v,x_1,x_2,x_3,x_4\}\rangle$ is isomorphic to the graph $K_1+C_4$.  Then, the vertex partition $\Pi=\{\{v,x_1,x_2\},\{x_3\},\{x_4\},\{u_1\},\{u_2\},...,\{u_{n-5}\}\}$ is a strong resolving partition for $G$ of cardinality $n-2$.
\end{proof}

\begin{theorem}
Let $G$ be a connected graph of order $n$. Then $pd_s(G)=n-1$ if and only if $G\cong P_3$, $G\cong C_4$, $G\cong K_n-e$\footnote{$K_n-e$ is the graph obtained from $K_n$ by deleting one edge.} or $G\cong K_1+\bigcup_{i}K_{n_i}$, $i>1$, $n_i\ge 1$ for every $i$ and $\sum_{i}n_i=n-1$.
\end{theorem}

\begin{proof}
Suppose $pd_s(G)=n-1$. By Theorem \ref{bound-cover-pd}, we have that $\alpha(G_{SR})\ge n-2$. Also, by Corollary \ref{bound-n-diam} we have that $D(G)=2$. We consider the following cases.\\

\noindent
{\bf Case 1.} The order of $G_{SR}$ is $|\partial(G)|=n$. Hence, there exists two non adjacent vertices $u,v$ in $G_{SR}$ which means $u,v$ are not mutually maximally distant in $G$ and also there exist $u',v'$ such that $u,u'$ and $v,v'$ are two pairs of mutually maximally distant vertices. If $n=3$, then $G\cong P_3$. If $n=4$, then we can observe that $G\cong C_4$, $G\cong K_4-e$, $G\cong S_{1,3}$ (the star graph with tree leaves) or $G\cong K_1+(K_1\cup K_2)$. Now on in this case, we suppose $n\ge 5$. We consider the following subcases.\\

\noindent
{\bf Case 1.1:} $\Delta(G)<n-1$. Hence, for every vertex $x$ of $G$ there exists a vertex $z$ such that they are not adjacent. Moreover, since $D(G)=2$, for every two non adjacent vertices $x,z$, there exists a vertex $y$ such that $xyz$ form a shortest $x-z$ path. Now, for the vertex $y$ there exists $w\ne x,z$ such that $y\not\sim w$. Let $yy'w$ be a shortest $y-w$ path. We have the following cases.\\

\noindent
{\bf (a)} If $y'\ne x$ and $y'\ne z$, then we have that $w\sim x$ or $w\sim z$ or there exist $x',z'$ such that $wx'x$ and $wz'z$ form a $w-x$ path and a $w-z$ path respectively. So, we have either:
\begin{itemize}
\item $xyy'wx'x$ induce a cycle $C_5$, in which case the Remark \ref{C_5-o-sphere-n-2} leads to a contradiction.
\item $zyy'wz'z$ induce a cycle $C_5$, and again the Remark \ref{C_5-o-sphere-n-2} leads to a contradiction.
\item $xyy'wx$ induce a cycle $C_4$ and $zyy'wz$ induce a cycle $C_4$. Thus both cycles together induces a sphere graph $S_{2,3}$ and by Remark \ref{C_5-o-sphere-n-2} we have a contradiction.
\end{itemize}

\noindent
{\bf (b)} If $y'=x$, then $y'\ne z$. Since $D(G)=2$ we have that if $w\not\sim z$, then there exists $z'$ such that $wz'z$ form a shortest $w-z$ path. As a consequence, $xyzz'wx$ is an induced subgraph of $G$ isomorphic to $C_5$, in which case the Remark \ref{C_5-o-sphere-n-2} leads to a contradiction. Thus, $w\sim z$. Since $n\ge 5$, there exists at least a vertex $a\ne x,y,z,w$ in $G$. Let $\{u_1,u_2,....,u_{n-5}\}$ be the other vertices of $G$ different from $x,y,z,w,a$. We consider the following cases.
\begin{itemize}
\item The vertex $a$ is adjacent to every vertex $x,y,z,w$. In this case $\langle \{a,x,y,z,w\}\rangle$ is isomorphic to $K_1+C_4$ and by Remark \ref{C_5-o-sphere-n-2} we have a contradiction.
\item The vertex $a$ is adjacent to three vertices, say $x,z,w$. Hence, the vertex partition $\Pi=\{\{a,x\},\{w,z\},\{y\},\{u_1\},\{u_2\},...,\{u_{n-5}\}\}$ is a strong resolving partition for $G$ of cardinality $n-2$, a contradiction.
\item The vertex $a$ is adjacent to two adjacent vertices, say $x,y$. We have a contradiction again since the vertex partition $\Pi=\{\{z,y\},\{w,x\},\{a\},\{u_1\},\{u_2\},...,\{u_{n-5}\}\}$ is a strong resolving partition for $G$ of cardinality $n-2$.
\item The vertex $a$ is adjacent to two non adjacent, say $x,z$. So, the vertex partition $\Pi=\{\{w,x\},\{y,z\},\{a\},\{u_1\},\{u_2\},...,\{u_{n-5}\}\}$ is a strong resolving partition for $G$ of cardinality $n-2$, a contradiction.
\item The vertex $a$ is adjacent to only one vertex, say $x$. Also a contradiction, since the vertex partition $\Pi=\{\{y,z\},\{a,x\},\{w\},\{u_1\},\{u_2\},...,\{u_{n-5}\}\}$ is a strong resolving partition for $G$ of cardinality $n-2$.
\end{itemize}

\noindent
{\bf (c)} If $y'=z$, then $y'\ne x$ and we can proceed analogously to the above case (b) to obtain a contradiction.\\

\noindent
{\bf Case 1.2:} $\Delta(G)=n-1$. Let $u$ be a vertex such that $\delta(u)=n-1$. Since $|\partial(G)|=n$, there exists $u'$ such that $u,u'$ are mutually maximally distant, which means that $d(u',x)\le d(u',u)=1$ for every $x\in N(u)$. Thus, $u'\sim x$ for every $x\in N(u)$, which means that also $\delta(u')=n-1$. Since $n\ge 5$ there exist at least three different vertices $a,b,c\in N(u)-\{u'\}$. Let $\{u_1,u_2,...,u_{n-5}\}$ the other vertices of $G$. If $a\not\sim b$ and $b\not\sim c$, then the vertex partition $\Pi=\{\{a,u\},\{c,u'\},\{b\},\{u_1\},\{u_2\},\{u_{n-5}\}\}$ is a strong resolving partition for $G$ of cardinality $n-2$, which is a contradiction. Thus, at most two of the vertices $a,b,c$ are non adjacent. If $n=5$, then $G\cong K_5-e$. Suppose $n>5$. If there exist two different pairs of vertices $a,b$ and $c,d$ ($a\ne b\ne c\ne d$) such that $a\not\sim b$ and $c\not\sim d$ and $\{u_1,u_2,...,u_{n-6}\}$ are the other vertices of $G$, then the vertex partition $\Pi=\{\{a,u\},\{b\},\{c,u'\},\{d\},\{u_1\},\{u_2\},\{u_{n-6}\}\}$ is a strong resolving partition for $G$ of cardinality $n-2$, which is a contradiction. Thus, at most there exist one pair of vertices in $G$ being non adjacent and, as a consequence, we obtain that $G\cong K_n-e$.\\

\noindent
{\bf Case 2.} The order of $G_{SR}$ is $|\partial(G)|=n-1$. Hence $\alpha(G_{SR})\le n-2$ and we obtain that $\alpha(G_{SR})=n-2$. Thus $G_{SR}$ is isomorphic to $K_{n-1}$. So, there is only one vertex $v$ in $G$ which is not mutually maximally distant with any other vertex of $G$ and every pair $x,y$ of different vertices of $V(G)-\left\{v\right\}$ are mutually maximally distant between them. If $\delta(v)<n-1$, then there exists a vertex $u$ such that $v\not\sim u$. Let $vx_1x_2...x_nu$ be a shortest $v-u$ path. Since $u$ is mutually maximally distant with every vertex in $V(G)-\left\{v\right\}$, particularly, $u$ is mutually maximally distant with $x_1,x_2,...,x_n$ and this is a contradiction. Thus, $\delta(v)=n-1$ and, as a consequence, $G$ is isomorphic to a graph $K_1+\bigcup_{i}G_{n_i}$ where $K_1=\langle \left\{v\right\}\rangle$ and $G_{n_i}$ is a connected graph of order $n_i$ for every $i$. If $i=1$, then $G$ is a complete graph and $pd_s(G)=n$, a contradiction. So, $i>1$. Now, suppose there exists $n_j$ such that $G_{n_j}$ is not a complete graph. Hence, there exist two different vertices $a,b$ in $G_{n_j}$ such that $a\not\sim b$. Notice that $d(a,b)=2$. Since $G_{n_j}$ is connected, there exists a shortest $a-b$ path, say $ayb$, in $G_{n_j}$. Also, as every two vertices in $V(G)-\{v\}$ are mutually maximally distant between them, we have that $a,y,b$ are mutually maximally distant between them, which leads to a contradiction. Therefore, $G$ is isomorphic to a graph $K_1+\bigcup_{i}G_{n_i}$ where $K_1=\langle \{v\}\rangle$ and $G_{n_i}$ is a complete graph of order $n_i$ for every $i$.
\end{proof}

\begin{theorem}\label{path-only-two}
Let $G$ be a connected graph. Then $pd_s(G)=2$ if and only if $G$ is a path.
\end{theorem}

\begin{proof}
If $G$ is a path, then $dim_s(G)=1$. Since $pd_s(G)\ge 2$ for every graph, by Theorem \ref{bound-dim-pd} we have that $pd_s(G)=2$. On the contrary, suppose that $pd_s(G)=2$ and let $\Pi=\left\{U_1,U_2\right\}$ be a strong partition basis of $G$. Let $x,y$ be two vertices belonging to the same set of the partition, say $x,y\in U_1$. If $d(x,U_2)=d(y,U_2)$, then $x,y$ are not strongly resolved by $\Pi$, which is a contradiction. Thus, $d(u,U_2)\ne d(v,U_2)$ for every pair of vertices of $U_1$. Analogously $d(u,U_1)\ne d(v,U_1)$ for every pair of vertices of $U_2$. Since $G$ is connected there exists only one vertex of $U_1$ adjacent to a vertex of $U_2$ and viceversa. So, the distances between vertices of $U_2$ and the set $U_1$ take every possible values in the set $\left\{1,...,|U_2|\right\}$ and, analogously, the distances between vertices of $U_1$ and the set $U_2$ take every possible values in the set $\left\{1,...,|U_1|\right\}$. Therefore $G$ is a path.
\end{proof}

As a consequence of the above characterization, if $G$ is a graph different from a path such that $\omega(G_{SR})=2$, then Corollary \ref{coro-pd-clique} can be improved at least by one.

\begin{remark}\label{pd-clique-3}
If $G$ is a connected graph different from the path graph such that $\omega(G_{SR})=2$, then $pd_s(G)\ge \omega(G_{SR})+1$.
\end{remark}

\begin{proof}
For any connected graph $G$, $pd_s(G)\ge \omega(G_{SR})$. Since $pd_s(G)=2$ if and only if $G$ is a path, it follows that $pd_s(G)\ge 3=\omega(G_{SR})+1$.
\end{proof}

In next results we give the exact value for the strong partition dimension of some families of graphs.

\begin{proposition}\label{pd-cycle}
For any cycle graph $C_n$, $pd_s(C_n)=3$.
\end{proposition}

\begin{proof}
Since the strong resolving graph of a cycle is either a cycle (if $n$ is odd) or a union of $n/2$ disjoint copies of $K_2$ (if $n$ is even), we have that $\omega((C_n)_{SR})=2$. So Remark \ref{pd-clique-3} leads to $pd_s(C_n)\ge 3$. On the other hand, let $V=\left\{v_0,v_1,...,v_{n-1}\right\}$ be the vertex of $C_n$, where two consecutive vertices (modulo $n$) are adjacent. Since $\Pi=\left\{\left\{v_0\right\},\left\{v_1,v_2,...,v_{\left\lfloor\frac{n}{2}\right\rfloor}\right\},
\left\{v_{\left\lfloor\frac{n}{2}\right\rfloor+1},...,v_{n-1}\right\}\right\}$ is a $\mathcal{P}_2$-partition for $C_n$ of cardinality three, by Remark \ref{P_2-partition} the result follows.
\end{proof}

We recall that the \emph{Cartesian product of two graphs} $G=(V_1,E_1)$ and $H=(V_2,E_2)$ is the graph $G\Box H$, such that $V(G\Box H)=V_1\times V_2$ and two vertices $(a,b),(c,d)$ are adjacent in  $G\Box H$ if and only if, either ($a=c$ and $bd\in E_2$) or ($b=d$ and $ac\in E_1$). Next we study the particular cases of grid graphs which are obtained as the Cartesian product of two paths.

\begin{theorem}
For any grid graph $P_m\Box P_n$ with $m,n\ge 2$, $pd_s(P_m\Box P_n)=3$.
\end{theorem}

\begin{proof}
Let $V_1=\left\{u_1,u_2,...,u_m\right\}$ and $V_2=\left\{v_1,v_2,...,v_n\right\}$ be the vertex sets of $P_m$ and $P_n$, respectively. Since for any two vertices $(u_i,v_j),(u_l,v_k)\in V_1\times V_2$, $d_{P_m\Box P_n}((u_i,v_j),(u_l,v_k))=d_{P_m}(u_i,u_l)+d_{P_n}(v_j,v_k)$ we have that $(u_i,v_j),(u_l,v_k)$ are mutually maximally distant in $P_m\Box P_n$ if and only if $u_i,u_l$ are mutually maximally distant in $P_m$ and $v_j,v_k$ are mutually maximally distant in $P_n$. So, we have that $\partial(P_m\Box P_n)=\left\{(u_1,v_1),(u_1,v_n),(u_m,v_1),(u_m,v_n)\right\}$ and $(P_m\Box P_n)_{SR}\cong \bigcup_{i=1}^2 K_2$. Thus $\omega((P_m\Box P_n)_{SR})=2$ and by Remark \ref{pd-clique-3} it follows $pd_s(P_m\Box P_n)\ge 3$.

Now, let $\Pi=\left\{\left\{(u_1,v_1)\right\},\left\{(u_1,v_n)\right\},(V_1\times V_2)-\left\{(u_1,v_1),(u_1,v_n)\right\}\right\}$ be a vertex partition of $P_m\Box P_n$. We shall prove that $\Pi$ is a strong resolving partition for $P_m\Box P_n$. Let $(u_i,v_j),(u_l,v_k)$ be two different vertices of $P_m\Box P_n$. We consider the following cases.

\noindent
Case 1. $j=k$. Hence, without loss of generality we suppose $i>l$. So, it is satisfied that
$$d_{P_m\Box P_n}((u_i,v_j),(u_1,v_1))=d_{P_m\Box P_n}((u_i,v_j),(u_l,v_k))+d_{P_m\Box P_n}((u_l,v_k),(u_1,v_1)),$$
and also,
$$d_{P_m\Box P_n}((u_i,v_j),(u_1,v_n))=d_{P_m\Box P_n}((u_i,v_j),(u_l,v_k))+d_{P_m\Box P_n}((u_l,v_k),(u_1,v_n)).$$
Thus the sets $\left\{(u_1,v_1)\right\}$ and $\left\{(u_1,v_n)\right\}$ strongly resolve $(u_i,v_j),(u_l,v_k)$.

\noindent
Case 2. $j\ne k$. Hence, without loss of generality let $j<k$. If $i\le l$, then we have that
$$d_{P_m\Box P_n}((u_l,v_k),(u_1,v_1))=d_{P_m\Box P_n}((u_l,v_k),(u_i,v_j))+d_{P_m\Box P_n}((u_i,v_j),(u_1,v_1))$$
and the set $\left\{(u_1,v_1)\right\}$ strongly resolves $(u_i,v_j),(u_l,v_k)$. On the contrary, if $i>l$, then we have that
$$d_{P_m\Box P_n}((u_i,v_j),(u_1,v_n))=d_{P_m\Box P_n}((u_i,v_j),(u_l,v_k))+d_{P_m\Box P_n}((u_l,v_k),(u_1,v_n))$$
and the set $\left\{(u_1,v_n)\right\}$ strongly resolves $(u_i,v_j),(u_l,v_k)$.

Therefore, $\Pi$ is a strong resolving partition for $P_m\Box P_n$ and the result follows.
\end{proof}

The {\em wheel graph} $W_{1,r}$ (respectively {\em fan graph} $F_{1,r}$) is the graph obtained from the graphs $K_1$ and $C_r$ (respectively $P_r$) by adding all possible edges between the vertices of $C_r$ (respectively $P_r$) and the vertex of $K_1$. Notice that $W_{1,r}=K_1+C_r$ and $F_{1,r}=K_1+P_r$. Next we obtain the strong partition dimension of $W_{1,r}$ and $F_{1,r}$. The vertex of $K_1$ is called the {\em central vertex} of the wheel or the fan.

\begin{remark}
For any wheel graph $W_{1,r}$, $r\ge 4$,
$$pd_s(W_{1,r})=\left\{\begin{array}{ll}
                           3, & \mbox{if $r=4$}, \\
                            &  \\
                           \left\lceil\frac{r}{2}\right\rceil, & \mbox{if $r\ge 5$}. \\
                         \end{array}\right.
$$
\end{remark}

\begin{proof}
Let $u$ be the central vertex and let $C=\left\{v_0,...,v_{r-1}\right\}$ be the set of vertices of the cycle used to construct $W_{1,r}$ (for every $i\in\left\{0,...,r-1\right\}$, $v_i\sim v_{i+1}$ where the operations with the subscripts are done modulo $r$). By doing simple calculation it is possible to check that if $r=4$, then $pd_s(W_{1,r})=3$. Now on we assume $r\ge 5$. Since the diameter of $W_{1,r}$ is two, any two non adjacent vertices of $C$ are mutually maximally distant between them. So, any two non adjacent vertices of $C$ belong to different sets of any strong resolving partition for $W_{1,r}$. Thus, every set of every strong resolving partition for $W_{1,r}$ must contain at most two vertices of $C$ and, as a consequence, $pd_s(W_{1,r})\ge \left\lceil\frac{r}{2}\right\rceil$.

Now, let the vertex partition $\Pi=\left\{\left\{u,v_0,v_1\right\},\left\{v_2,v_3\right\},...,\left\{v_{r-2},v_{r-1}\right\}\right\}$ if $r$ is even or $\Pi=\left\{\left\{u,v_0,v_1\right\},\left\{v_2,v_3\right\},...,\left\{v_{r-3},v_{r-2}\right\},\left\{v_{r-1}\right\}\right\}$ if $r$ is odd. It is straightforward to check that $\Pi$ is a strong resolving partition for $W_{1,r}$. Therefore $pd_s(W_{1,r})\le \left\lceil\frac{r}{2}\right\rceil$ and the result follows.
\end{proof}

By using similar arguments we obtain the strong metric dimension of fan graphs.

\begin{remark}
For any fan graph $F_{1,r}$, $r\ge 3$,
$$pd_s(F_{1,r})=\left\{\begin{array}{ll}
                           3, & \mbox{if $r=3,4$}, \\
                            &  \\
                           \left\lceil\frac{r}{2}\right\rceil, & \mbox{if $r\ge 5$}. \\
                         \end{array}\right.
$$
\end{remark}

\section{Strong partition dimension of unicyclic graphs}

From now on we will denote by $G(C_t)$ the unicyclic graph different from a cycle whose unique cycle $C_t$ has vertex set $u_0,u_1,...,u_{t-1}$ with $t\ge 3$ an $u_i\sim u_{i+1}$ (operations with the subindex $i$ are done modulo $t$), for every $i\in \left\{0,...,t-1\right\}$.

\begin{theorem}
Let $G(C_t)$, $t\ge 3$, be a unicyclic graph of order $n$. If $|\tau(G(C_t))|=1$, then $pd_s(G(C_t))=3$.
\end{theorem}

\begin{proof}
If $|\tau(G)|=1$, then $G(C_t)$ has only one major vertex. Without loss of generality we suppose that the major vertex is $u_0$. Let $v_0$ be the terminal vertex of $u_0$ and let $P(u_0,v_0)$ be the shortest $u_0-v_0$ path in $G(C_t)$. Now, let the vertex partition $\Pi=\left\{A_1,A_2,A_3\right\}$ such that
$$A_1=V(P(u_0,v_0))\cup \left\{u_1,u_2,...,u_{\left\lfloor\frac{t}{2}\right\rfloor-1}\right\},\;A_2\left\{u_{\left\lfloor\frac{t}{2}\right\rfloor}\right\},\; A_3=\left\{u_{\left\lfloor\frac{t}{2}\right\rfloor+1},...,u_{t-1}\right\}\mbox{ if $t$ is even or,}$$
$$A_1=V(P(u_0,v_0))\cup \left\{u_1,u_2,...,u_{\left\lfloor\frac{t}{2}\right\rfloor-1}\right\},\; A_2=\left\{u_{\left\lfloor\frac{t}{2}\right\rfloor}, u_{\left\lceil\frac{t}{2}\right\rceil}\right\},\;A_3=\left\{u_{\left\lceil\frac{t}{2}\right\rceil+1},...,u_{t-1}\right\}\mbox{ if $t$ is odd.}$$
We claim that $\Pi$ is a strong resolving partition for $G(C_t)$. We consider two different vertices $x,y$ of $G(C_t)$.
If $x,y\in A_1$ or $x,y\in A_3$, then since $u_0,u_{\left\lfloor\frac{t}{2}\right\rfloor}$ (for $t$ even) and $u_0,u_{\left\lceil\frac{t}{2}\right\rceil}$ and $u_0,u_{\left\lfloor\frac{t}{2}\right\rfloor}$ (for $t$ odd) are diametral vertices in $C_t$ we have either $d(x,A_2)=d(x,y)+d(y,A_2)$ or $d(y,A_2)=d(y,x)+d(x,A_2)$. Also, if $t$ is odd and $x,y\in A_2$, then either $d(x,A_i)=d(x,y)+d(y,A_i)$ or $d(y,A_i)=d(y,x)+d(x,A_i)$ with $i\in \left\{1,3\right\}$.
Therefore, $\Pi$ is a strong resolving partition for $G(C_t)$ and we have that $pd_s(G(C_t))\le 3$. Finally the results follows by Theorem \ref{path-only-two}.
\end{proof}

\begin{theorem}\label{teo-unicyclic}
Let $G(C_t)$, $t\ge 3$, be a unicyclic graph of order $n$. If $|\tau(G(C_t))|\ge 2$, then
$$|\tau(G(C_t))|\le pd_s(G(C_t))\le |\tau(G(C_t))|+2.$$
\end{theorem}

\begin{proof}
If $t=3$, then $G(C_3)$ is a block graph. Let $q$ be the number of vertices of $C_3$ being major vertices in $G(C_3)$. So, $G(C_3)$ has $n-|\tau(G(C_3))|-3+q$ cut vertices. Thus, by Theorem \ref{easy-results} (iii) we have that $pd_s(G(C_3))=n-(n-|\tau(G(C_3))|-3+q)=|\tau(G(C_3))|+3-q$. Since $1\le q\le 3$ we obtain that $|\tau(G(C_3))|\le pd_s(G(C_3))\le |\tau(G(C_3))|+2$. Now on we suppose $t\ge 4$. Notice that any two different vertices of $\tau(G(C_t))$ are mutually maximally distant between them. Thus, $\omega(G(C_t))\ge |\tau(G(C_t))|$. Thus by Theorem \ref{coro-pd-clique} we have that $pd_s(G(C_t))\ge\omega(G_{SR})=|\tau(G(C_t))|$. Now, without loss of generality suppose that $u_0$ is a major vertex and for any major vertex $u_i$ of $C_t$ let $v_{ij}$, $j\in \left\{1,...,ter(u_i)\right\}$, be a terminal vertex of $u_i$. For every major vertex $u_i$ and all its terminal vertices let $P(u_i,v_{ij})$ be the shortest $u_i-v_{ij}$ path in $G(C_t)$.

Now, for the major vertex $u_0$ and all its terminal vertices $v_{0j}$, $j\in \left\{1,...,ter(u_0)\right\}$ we define the following sets
\begin{align*}
A_{0,1}&=V(P(u_0,v_{01}))\cup \left\{u_0,u_1,...,u_{\left\lfloor\frac{t}{2}\right\rfloor-1}\right\}\\
A_{0,2}&=V(P(u_{0},v_{02}))-A_{0,1}-\left\{u_0\right\}\\
A_{0,3}&=V(P(u_{0},v_{03}))-A_{0,2}-A_{0,1}-\left\{u_0\right\}\\
& \hspace*{0.7cm}.................................................\\
A_{0,ter(u_0)}&=V(P(u_{0},v_{0\,ter(u_0)}))-A_{0,ter(u_0)-1}-...-A_{0,1}-\left\{u_0\right\}
\end{align*}
Moreover, for every major vertex $u_i$, $i\ne 0$, and all its terminal vertices $v_{ij}$, $j\in \left\{1,...,ter(u_i)\right\}$ we define the following sets.
\begin{align*}
A_{i,1}&=V(P(u_i,v_{i1}))-\left\{u_i\right\}\\
A_{i,2}&=V(P(u_{i},v_{i2}))-A_{i,1}-\left\{u_i\right\}\\
A_{i,3}&=V(P(u_{i},v_{i3}))-A_{i,2}-A_{i,1}-\left\{u_i\right\}\\
& \hspace*{0.7cm}.................................................\\
A_{i,ter(u_i)}&=V(P(u_{i},v_{i\,ter(u_i)}))-A_{i,ter(u_i)-1}-...-A_{i,1}-\left\{u_i\right\}
\end{align*}
Also let $B=\left\{u_{\left\lfloor\frac{t}{2}\right\rfloor}\right\}$ if $t$ is even or $B=\left\{u_{\left\lfloor\frac{t}{2}\right\rfloor},u_{\left\lceil\frac{t}{2}\right\rceil}\right\}$ if $t$ is odd, and let $C=\left\{u_{\left\lceil\frac{t}{2}\right\rceil+1},...,u_{t-1}\right\}$. Notice that the sets $B$, $C$ and the sets $A_{i,j}$ defined for every major vertex $u_i$ and all its terminal vertices $v_{ij}$ form a vertex partition $\Pi$ of $G(C_t)$ of cardinality $|\tau(G(C_t))|+2$. An example of the partition is drawn in Figure \ref{proof-unicyc-2}.
\begin{figure}[h]
  \centering
  \includegraphics[width=0.5\textwidth]{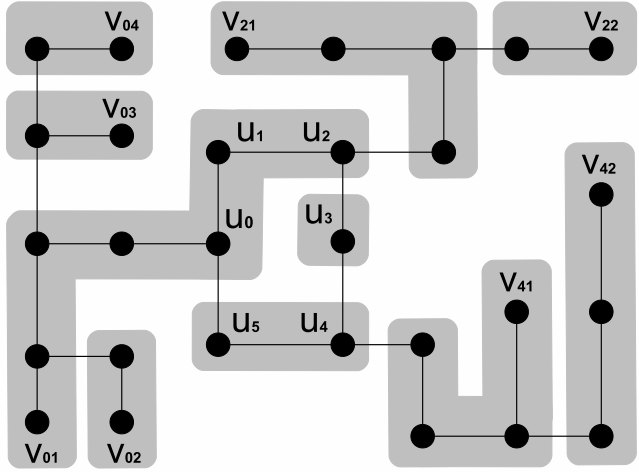}
  \caption{An example of the partition $\Pi$.}\label{proof-unicyc-2}
\end{figure}
We claim that the vertex partition $\Pi$ is a strong resolving partition for $G(C_t)$. Let $x,y$ be two different vertices of $G(C_t)$. We consider the following.
\begin{itemize}
\item If $x,y\in A_{0,1}$ or $x,y\in C$, then since $u_0,u_{\left\lfloor\frac{t}{2}\right\rfloor}$ (for $t$ even) and $u_0,u_{\left\lceil\frac{t}{2}\right\rceil}$ and $u_0,u_{\left\lfloor\frac{t}{2}\right\rfloor}$ (for $t$ odd) are diametral vertices in $C_t$ we have either $d(x,B)=d(x,y)+d(y,B)$ or $d(y,B)=d(y,x)+d(x,B)$.
\item If $x,y\in A_{i,j}$ related to some major vertex $u_i$ and $j\in \left\{2,...,ter(u_i)\right\}$, then we have either $d(x,A_{i,j-1})=d(x,y)+d(y,A_{i,j-1})$ or $d(y,A_{i,j-1})=d(y,x)+d(x,A_{i,j-1})$.
\item If $x,y\in A_{i,1}$ related to some major vertex $u_i\ne u_0$, then we have the following cases.
\begin{itemize}
\item If $i\in \left\{1,...,\left\lfloor\frac{t}{2}\right\rfloor-1\right\}$, then we have either $d(x,A_{0,1})=d(x,y)+d(y,A_{0,1})$ or $d(y,A_{0,1})=d(y,x)+d(x,A_{0,1})$.
\item If $i\in \left\{\left\lfloor\frac{t}{2}\right\rfloor,\left\lceil\frac{t}{2}\right\rceil\right\}$, then we have either $d(x,B)=d(x,y)+d(y,B)$ or $d(y,B)=d(y,x)+d(x,B)$.
\item If $i\in \left\{\left\lceil\frac{t}{2}\right\rceil+1,...,t-1\right\}$, then we have either $d(x,C)=d(x,y)+d(y,C)$ or $d(y,C)=d(y,x)+d(x,C)$.
\end{itemize}
\item If $t$ is odd and $x,y\in B$, then we have either $d(x,C)=d(x,y)+d(y,C)$ or $d(y,C)=d(y,x)+d(x,C)$.
\end{itemize}
Therefore, $\Pi$ is a strong resolving partition for $G(C_t)$ and the proof is complete.
\end{proof}

The above bounds are tight as we can see at next. Moreover, there are unicyclic graphs achieving also the only value in the middle between lower and upper bound. An example of that is the family of unicyclic graphs $\mathcal{C}_1$ of Proposition \ref{prop-C_1}.

\begin{proposition}
Let $G(C_t)$, $t\ge 3$, be a unicyclic graph of order $n$ and $|\tau(G(C_t))|\ge 2$.
\begin{itemize}
\item If $t=3$ and $C_t$ has only one major vertex, then $pd_s(G(C_t))=|\tau(G(C_t))|+2$.
\item If every vertex of $C_t$ is a major vertex, then $pd_s(G(C_t))=|\tau(G(C_t))|$.
\end{itemize}
\end{proposition}

\begin{proof}
If $t=3$, then $G(C_3)$ is a block graph with $n-|\tau(G(C_3))|-2$ cut vertices. Thus by Theorem \ref{easy-results} (iii) we have that $pd_s(G(C_3))=n-(n-|\tau(G(C_3))|-2)=|\tau(G(C_3))|+2$ and (i) is proved.

The technique of the proof of (ii) is relative similar to the proof of Theorem \ref{teo-unicyclic}. Without loss of generality suppose that. As above for any major vertex $u_i$, $i\in \{0,...,t-1\}$, and all its terminal vertices $v_{ij}$, $j\in \left\{1,...,ter(u_i)\right\}$ we define the following sets
\begin{align*}
A_{i,1}&=V(P(u_i,v_{i1}))\\
A_{i,2}&=V(P(u_{i},v_{i2}))-A_{i,1}\\
A_{i,3}&=V(P(u_{i},v_{i3}))-A_{i,2}-A_{i,1}\\
& \hspace*{0.7cm}.................................................\\
A_{i,ter(u_i)}&=V(P(u_{i},v_{i\,ter(u_i)}))-A_{i,ter(u_i)-1}-...-A_{i,1}
\end{align*}
Notice that the sets $A_{i,j}$ defined for every major vertex $u_i$ and all its terminal vertices $v_{ij}$ form a vertex partition $\Pi$ of $G(C_t)$ of cardinality $|\tau(G(C_t))|$.
By using similar ideas to those ones in the the proof of Theorem \ref{teo-unicyclic} we obtain that the vertex partition $\Pi$ is a strong resolving partition for $G(C_t)$. Therefore $pd_s(G(C_t))\le |\tau(G(C_t))|$ and the result follows by the lower bound of Theorem \ref{teo-unicyclic}.
\end{proof}

\end{document}